\documentclass[a4paper,12pt]{article}
\usepackage{amsmath,fancyhdr,amssymb,amsthm, geometry, enumerate, graphicx ,amsfonts,hyperref,mathrsfs}
\title{ $\beta$-$n$-sensitivity and other stronger forms of sensitivity in dynamical systems}
\author{ Hongbo Zeng$^{a*}$}

\title{ $\beta$-$n$-sensitivity and other stronger forms of sensitivity in dynamical systems}
\theoremstyle{definition}

\providecommand{\keywords}[1]{\textbf{Keywords :} #1}

\theoremstyle{plain}

\newtheorem{definition}{Definition}
\newtheorem{lemma}{Lemma}
\newtheorem{remark}{Remark}

\newtheorem{theorem}{Theorem}
\newtheorem{example}{Example}
\newtheorem{corollary}{Corollary}
\begin{document}
\date{}
\maketitle

\begin{abstract}
In this paper, we investigate the relationships between the $\beta$-$n$-sensitivity and other forms of sensitivity in discrete dynamical systems. And we present some different sufficient conditions (equivalent condition) to be (strongly) $\beta$-$n$-sensitive. These results improve and extend some existing ones. Moreover, we introduce and study the semi-strong $\beta$-$n$-sensitivity and (strong) multi-$\beta$-$n$-sensitivity with respect to $\alpha$.  Besides, we proved that strongly multi-transitivity implies strongly multi-sensitivity, which improves some existing ones. Finally, we obtain that (semi-)strongly $\beta$-$n$-sensitivity is preserved under iterations but $\beta$-$n$-sensitivity fails.

\end{abstract}


\bigskip\renewcommand{\thefootnote}{\fnsymbol{footnote}}
\footnotetext{\hspace*{-5mm}
\renewcommand{\arraystretch}{1}
\begin{tabular}{@{}r@{}p{11cm}@{}}
$^*$& the corresponding author. \emph{Email addresses}: zenghongbo@csust.edu.cn.\\
$^a$&School of Mathematics and Statistics, Changsha University of Science and Technology, Changsha, P.R. China.\\


\end{tabular}}
\vspace{-2mm}
\section{Introduction}

Dynamical systems theory is an effective mathematical mechanism which describes the time dependence of a point in a geometric space and has remarkable connections with different areas of mathematics such as topology and number theory. It is used to deal with the complexity, instability, or chaoticity in the real world, such as in meteorology, ecology, celestial mechanics, and other natural sciences. In recent years, more and more scholars have begun to devote themselves to the research in topological dynamical systems, and have achieved many significant results (see \cite{q12,q7}). The chaos theory is one of the most important components of dynamical systems, which was first strictly defined by Li and Yorke in 1975 \cite{a1a}. Since then, different people from different fields gave different definitions of chaos under their understanding of the subject. In 1986, Devaney proposed the widely accepted definition of chaos \cite{a3a}. Later, Banks et al found that in the three conditions defining Devaney chaos, topological transitivity and dense periodic points together imply sensitivity \cite{q8}.  Sensitive dependence on initial conditions(briefly, sensitivity), first defined in \cite{q9}, is also one of the most remarkable components of dynamical systems theory, which is closely linked to different variants of chaos. It characterizes the unpredictability in chaotic phenomena and is an integral part of different types of chaos. Afterwards, Li-Yorke sensitivity \cite{q10} and some stronger forms of sensitivity (including cofinite sensitivity, multi-sensitivity, and syndetic sensitivity) \cite{q11} were successively proposed. And various related
yet distinct definitions of sensitivity have been proposed from different perspectives. For
instance, from the perspective of multi-variable versions, there exist $n$-sensitivity \cite{b1,b2}, thick $n$-sensitivity and block thick $n$-sensitivity \cite{b3}. From the standpoint of trajectory iteration step lengths, there are $(r,s)$-sensitivity \cite{b4} and $(r,s)$-Li-Yorke sensitivity \cite{b5}. Recently, Yang \cite{a1}  provide a broader definition of sensitivity, namely strong multi-$\beta$-$n$-sensitivity and $\beta$-$n$-sensitivity.

 Currently, sensitivity and stronger forms of sensitivity are widely studied in dynamical systems  \cite{b6,b7,b8,b9,b10,b11,b12,b13,b14,b15}. For example, Mai \cite{b10} described sensitivity with asymptotically almost periodic points, uniformly segment-recurrent points and allured sets. Fedeli \cite{b6} proved that if  $f$ is a weakly mixing self-map on a space  X which is disconnected or a Urysohn space with at least two points, then  $f$ is topologically sensitive. Huang \cite{b7} showed that a minimal system is either multi-sensitive or an almost one-to-one extension of its maximal equicontinuous
factor. Li \cite{b9} obtained sufficient and necessary conditions for thick n-sensitivity and
blockily thick $n$-sensitivity for minimal group actions. Mahajan \cite{b11} obtained some sufficient conditions for a dynamical system to be  F -sensitive and for a semiflow to be multi-sensitive with respect to a vector.  Yang \cite{a1} proved that chain mixing systems with the shadowing property exhibit
strong strong multi-$\beta$-$n$-sensitivity and strong $\beta$-$n$-sensitivity under the condition of surjection.

Motivated by the above results, this paper is devoted to further study the properties of $\beta$-$n$-sensitivity and other stronger forms of sensitivity in dynamical systems.

This paper is organized as follows. In Section 2, we will first state some preliminaries, definitions and some lemmas. The main conclusions will be given in Section 3. This paper is concerned with the relationships between the $\beta$-$n$-sensitivity and other forms of sensitivity in dynamical systems.  we give some different sufficient conditions  to be (strongly) $\beta$-$n$-sensitive and some equivalent conditions to be cofinitely strongly $\beta$-$n$-sensitive (thickly syndetically strongly $\beta$-$n$-sensitive).  Meanwhile, we introduce and study the semi-strong $\beta$-$n$-sensitivity and (strong) multi-$\beta$-$n$-sensitivity with respect to $\alpha$.  Besides, we proved that strongly multi-transitivity implies strongly multi-sensitivity, which improves the result in  \cite{b16}. Finally, we obtain that (semi-)strongly $\beta$-$n$-sensitivity is preserved under iterations but $\beta$-$n$-sensitivity fails.

\section{Preliminary}

In this section, we mainly give some different concepts of sensitivity (see, for example, \cite{q11,b4,a1,b16}) and some lemmas required for remaining sections of the paper.

Assume that $\mathbb{N}=\{1,2,3,...\}$ . Let $(X,d)$ be a metric
space and $f: X \rightarrow X$ be a continuous function. An  discrete
dynamical system is a pair $(X, f)$. For any $n\in \mathbb{N}$,
define the composition
$f^n:= f \circ \cdot\cdot\cdot \circ f$. The orbit of
a point $x$ in $X$ is the set
$orb(x, f):=\{x,f^1(x),f^2(x),...,f^n(x),...\}$. A set $A\subseteq \mathbb{N}$ is called syndetic if there exists a positive integer $M$ such that $\{i,i+1,...,i+M\}\cap A\neq\emptyset$ for every $i\in \mathbb{N}$, i.e. it has bounded gaps. A set $A\subseteq \mathbb{N}$ is called cofinite if there exists $N\in \mathbb{N}$ such that $A\supseteq[N,\infty)\cap \mathbb{N}$. A set $A\subseteq \mathbb{N}$ is called thick if it contains arbitrarily long runs of positive integers, that is, for any $p\in \mathbb{N}$, there exists some $n\in \mathbb{N}$ such that $\{n,n+1,...,n+p\}\subseteq A$. A set $A\subseteq \mathbb{N}$ is called thickly syndetic if $\{m\in \mathbb{N}: m+j\in A, \text{for } 0\leq j\leq l\}$ is syndetic for any $l\in \mathbb{N}$. In this paper, we always suppose that $(X, f)$  is a discrete dynamical system and that $n$ is a positive integer with $n\geq2$ unless otherwise specified.

Let $U,V\subseteq X$, $\delta>0$, $r,s\in \mathbb{N}$ and $\beta=(m_1,m_2,...,m_n)\in \mathbb{N}^n$. Denote by 
$$N_{f}(U,V)=\{k\in\mathbb{N}\mid f^(U)\cap V\neq\emptyset\},$$
$$N_{f}(U,r,s,\delta)=\{k\in\mathbb{N}\mid \exists x,y\in U  \text{ such that }  d(f^{rk}(x),f^{sk}(y))>\delta\},$$
$$N_{f}(U,n,\delta)=\{k\in\mathbb{N}\mid \exists x_1,x_2,...,x_n\in U  \text{ such that }  \min_{1\leq i<j\leq n}d(f^{k}(x_i),f^{k}(x_j))>\delta\},$$
$$N_{f}(U,\beta,\delta)=\{k\in\mathbb{N}\mid \exists x_1,x_2,...,x_n\in U  \text{ such that }  \min_{1\leq i<j\leq n}d(f^{m_ik}(x_i),f^{m_jk}(x_j))>\delta\}.$$

\begin{definition}

A topological dynamical system $(X, f)$  is said to be (topologically) transitive, if for any two nonempty subsets $U,V\subseteq X$, $N_{f}(U,V)$  is nonempty. Equivalently, if for any two nonempty subsets $U,V\subseteq X$, there exists $n\in \mathbb{N}$ such that $f_1^n(U)\cap V\neq\emptyset$.

A topological dynamical system $(X, f)$  is said to be weakly mixing, if the product system $(X^2,f^{2})$ is transitive. Equivalently, if for any four nonempty subsets $U_1,U_2,V_1,V_2\subseteq X$, there exists $n\in \mathbb{N}$ such that $f^n(U_1)\cap V_1\neq\emptyset$ and $f^n(U_2)\cap V_2\neq\emptyset$.

A topological dynamical system $(X, f)$ is said to be mixing, if for any two nonempty subsets $U,V\subseteq X$, $N_{f}(U,V)$  is cofinite. Equivalently, if for any two non-empty subsets $U,V\subseteq X$, there exists $N\in \mathbb{N}$ such that $f_1^n(U)\cap V\neq\emptyset$ for any $n\geq N$.

A topological dynamical system $(X, f)$   is said to be totally transitive, if for any $n\in \mathbb{N}$, $f^{n}$ is transitive.

Let $\alpha=(a_1,a_2,...,a_r)\in\mathbb{N}^r$.  A topological dynamical system $(X, f)$   is said to be multi-transitive with respect to $\alpha$, if the product system $(X^r,f^{(\alpha)})$ is transitive, where $f^{(\alpha)}=f^{a_1}\times f^{a_2}\times\cdot\cdot\cdot f^{a_r}$. 

A topological dynamical system $(X, f)$ is said to be multi-transitive if it is multi-transitive with respect to $(1,2,...,n)$ for any $n\in\mathbb{N}$. Equivalently, if for any $m\in \mathbb{N}$ and for any collection of nonempty open subsets $U_1,U_2,...,U_m;V_1,V_2,...,V_m$ of $X$, there exists $l\in \mathbb{N}$ such that  $f^{il}(U_i)\cap V_i\neq\emptyset$ for each $i\in\{1,2,...,m\}$.

A topological dynamical system $(X, f)$ is said to be strongly multi-transitive if it is multi-transitive with respect to any vector in $\mathbb{N}^n$ and any $n\in\mathbb{N}$.

A topological dynamical system $(X, f)$ is said to be minimal if the orbit of every $x\in X$ is dense in $X$.

\end{definition}

\begin{definition}
A topological dynamical system $(X, f)$ is said to be sensitive if there is a constant  $\delta>0$ such that for any nonempty open subset $U$ of $X$, $N_{f}(U,2,\delta)$ is nonempty.

A topological dynamical system $(X, f)$ is said to be cofinitely sensitive (syndeticly sensitive, thickly sensitive, thickly syndeticly sensitive, respectively) if there is a constant  $\delta>0$ such that for any nonempty open subset $U$ of $X$, $N_{f}(U,2,\delta)$ is  a syndetic set (syndetic set, thick set, thickly syndetic set, respectively).

\end{definition}

\begin{definition}
A topological dynamical system $(X, f)$ is said to be $n$-sensitive if there is a constant  $\delta>0$ such that for any nonempty open subset $U$ of $X$, $N_{f}(U,n,\delta)$ is nonempty. 

A topological dynamical system $(X, f)$ is said to be cofinitely $n$-sensitive (syndeticly $n$-sensitive, thickly $n$-sensitive, thickly syndeticly $n$-sensitive, respectively) if there is a constant  $\delta>0$ such that for any nonempty open subset $U$ of $X$, $N_{f}(U,n,\delta)$ is a syndetic set (syndetic set, thick set, thickly syndetic set, respectively).

A topological dynamical system $(X, f)$ is said to be multi-$n$-sensitive if there is a constant  $\delta>0$ such that for any $m\in\mathbb{N}$ and any nonempty open subsets $U_1,...,U_m$ of $X$, the set $\cap_{1\leq i\leq m}N_{f}(U_i,n,\delta)$ is nonempty.

\end{definition}

\begin{definition}
A topological dynamical system $(X, f)$ is said to be $(r,s)$-sensitive if there is a constant  $\delta>0$ such that for any nonempty open subset $U$ of $X$, $N_{f}(U,r,s,\delta)$ is nonempty.

A topological dynamical system $(X, f)$ is said to be cofinitely sensitive (syndeticly sensitive, thickly sensitive, thickly syndeticly sensitive, respectively) if there is a constant  $\delta>0$ such that for any nonempty open subset $U$ of $X$, $N_{f}(U,r,s,\delta)$ is a cofinite set (syndetic set, thick set, thickly syndetic set, respectively).

\end{definition}

\begin{definition}
Let $n\in\mathbb{N}$ with $n\geq2$. Let $\beta=(m_1,m_2,...,m_n)\in \mathbb{N}^n$. $(X,f)$ is said to be multi-sensitive with respect to $\beta$, if there exists $\delta>0$ such that for any nonempty open sets $U_1,...,U_n$ of $X$, $N_{f^{m_1}}(U_1,\delta)\cap N_{f^{m_2}}(U_2,\delta)\cap...\cap N_{f^{m_n}}(U_n,\delta)\neq\emptyset$, that is, $\cap_{1\leq i\leq n}N_{f^{m_i}}(U_i,\delta)\neq\emptyset$.	$(X,f)$ is said to be  $\mathcal{N}$ sensitive, if $(X,f)$ is multi-sensitive with respect to $(1,2,...,n)$ for any $n\in \mathbb{N}$. $(X,f)$ is said to be strongly multi-sensitive if it is multi-sensitive with respect to any vector in $\mathbb{N}^n$ and any $n\in\mathbb{N}$.
\end{definition}

\begin{definition}
Let $n\in\mathbb{N}$ with $n\geq2$. Let $\beta=(b_1,b_2,...,b_n)\in \mathbb{N}^n$. \\
A dynamical system $(X,f)$ is said to be $\beta$-$n$-sensitive, if there exists $\delta>0$ such that for any nonempty open set $U$ of $X$, $N_{f}(U,\beta,\delta)$ is nonempty. \\
A dynamical system $(X,f)$ is said to be strongly $\beta$-$n$-sensitive, if $(X,f)$ is $\beta$-$n$-sensitive for any vector $\beta\in\mathbb{N}^n$ and any $n\in\mathbb{N}$.\\
A dynamical system $(X,f)$ is said to be cofinitely $\beta$-$n$-sensitive  (syndeticly $\beta$-$n$-sensitive, thickly $\beta$-$n$-sensitive, thickly syndeticly $\beta$-$n$-sensitive, respectively), if there exists $\delta>0$ such that for any nonempty open set $U$ of $X$, $N_{f}(U,\beta,\delta)$ is cofinite set cofinite set (syndetic set, thick set, thickly syndetic set, respectively).
\end{definition}

\begin{definition}
Let $n\in\mathbb{N}$ with $n\geq2$. Let $\beta=(b_1,b_2,...,b_n)\in \mathbb{N}^n$. A dynamical system $(X,f)$ is said to be multi-$\beta$-$n$-sensitive, if there exists $\delta>0$ such that for any $m\in\mathbb{N}$ and any nonempty open sets $U_1,...,U_m$ of $X$, $N_{f}(U_1,\beta,\delta)\cap N_{f}(U_2,\beta,\delta)\cap...\cap N_{f}(U_n,\beta,\delta)\neq\emptyset$, that is, $\cap_{i=1}^mN_{f}(U_i,\beta,\delta)\neq\emptyset$.	Moreover, $(X,f)$ is said to be strongly multi-$\beta$-$n$-sensitive, if $(X,f)$ is  multi-$\beta$-$n$-sensitive for any vector $\beta\in\mathbb{N}^n$ and any $n\in\mathbb{N}$.
\end{definition}

We wish to formulate stronger notions of sensitivity in a similar way. Now, motivated
by this idea, we introduce the notions of the semi-strong $\beta$-$n$-sensitivity and (strong) multi-$\beta$-$n$-sensitivity with respect to $\alpha$.

\begin{definition}
Let $n\in\mathbb{N}$ with $n\geq2$. Let $\beta=(b_1,b_2,...,b_n)\in \mathbb{N}^n$. A dynamical system  $(X,f)$ is said to be semi-strongly $\beta$-$n$-sensitive, if $(X,f)$ is  $\beta$-$n$-sensitive for any vector $\beta\in\mathbb{N}^n$. A dynamical system  $(X,f)$ is said to be semi-strongly multi-$\beta$-$n$-sensitive, if $(X,f)$ is  multi-$\beta$-$n$-sensitive for any vector $\beta\in\mathbb{N}^n$.
\end{definition}

\begin{remark}
By definitions, strong $\beta$-$n$-sensitivity implies semi-strong $\beta$-$n$-sensitivity.  Conversely, there is a  minimal system $(X,f)$ which is $2$-sensitive but not $3$-sensitive \cite[Example 3.8]{b2}, further, it is not strong $\beta$-$n$-sensitive, on the other hand, by \cite[Theorem 5]{b4}, the  minimal system $(X,f)$ is also semi-strongly $\beta$-$2$-sensitive, which shows that semi-strongly $\beta$-$n$-sensitivity is strictly weaker than strongly $\beta$-$n$-sensitivity.
Besides, the following Example \ref{ex1} shows that there is a  dynamical system $(X,f)$ which is $(2,4)$-2-sensitive but not semi-strongly $\beta$-$n$-sensitive, which shows that semi-strongly $\beta$-$n$-sensitivity is strictly stronger than $\beta$-$n$-sensitivity.
\end{remark}

\begin{definition}
Let $n\in\mathbb{N}$ with $n\geq2$. Let $\beta=(b_1,b_2,...,b_n)\in \mathbb{N}^n$ and $\alpha=(a_1,a_2,...,a_n)\in \mathbb{N}^n$. A dynamical system $(X,f)$ is said to be multi-$\beta$-$n$-sensitive with respect to $\alpha$, if there exists $\delta>0$ such that for any $m\in\mathbb{N}$ and any nonempty open sets $U_1,...,U_m$ of $X$, $N_{f^{a_1}}(U_1,\beta,\delta)\cap N_{f^{a_2}}(U_2,\beta,\delta)\cap...\cap N_{f^{a_n}}(U_n,\beta,\delta)\neq\emptyset$, that is, $\cap_{i=1}^mN_{f^{a_i}}(U_i,\beta,\delta)\neq\emptyset$.	Moreover, $(X,f)$ is said to be strongly multi-$\beta$-$n$-sensitive with respect to $\alpha$, if $(X,f)$ is  multi-$\beta$-$n$-sensitive with respect to $\alpha$ for any vector $\beta,\alpha\in\mathbb{N}^n$ and any $n\in\mathbb{N}$.
\end{definition}

\begin{lemma}\label{yinli1}
	Let $n\in\mathbb{N}$ with $n\geq2$. Let $A_i=\{x_{i1},x_{i2},...,x_{in}\}$, where $i=1,2,...,n$. If for each $i\in\{1,2,...,n\}$,
$$\min_{1\leq j<l\leq n}d(f^{m_ik}(x_{ij}),f^{m_ik}(x_{il}))>a,$$
then there exist $y_i\in A_i$ for each $i\in\{1,2,...,n\}$ such that
$$\min_{1\leq j<l\leq n}d(f^{m_jk}(y_{j}),f^{m_lk}(y_{l}))>\frac{a}{2}.$$

\end{lemma}
\begin{proof}
Step 1. Take $y_1\in A_1$. Then since
$$\min_{1\leq j<l\leq n}d(f^{m_2k}(x_{2j}),f^{m_2k}(x_{2l}))> a,$$
there exist  at least $n-1$ points $\{z_{21},z_{22},...,z_{2(n-1)}\}$ of $A_2$ such that
\begin{equation}\label{lem1}
\begin{aligned}
d(f^{m_2k}(z_{2j}),f^{m_1k}(y_{1}))> \frac{a}{2}
\end{aligned}
\end{equation}
for every $j=1,2,...n-1$. If not, there exist two points $z_{21},z_{22}\in A_2$ with $$d(f^{m_2k}(z_{2j}),f^{m_1k}(y_{1}))\leq \frac{a}{2}$$
for every $j=1,2$. By the triangle inequality, we have
$$d(f^{m_2k}(z_{21}),f^{m_2k}(z_{22}))\leq d(f^{m_2k}(z_{21}),f^{m_1k}(y_{1}))+d(f^{m_2k}(z_{22}),f^{m_1k}(y_{1}))\leq a,$$
which is a contradiction. Now take $y_2\in \{z_{21},z_{22},...,z_{2(n-1)}\}$.

 Step 2. With the similar argument, we have that there exist at least $n-1$ points $\{z_{31},z_{32},...,z_{3(n-1)}\}$ of $A_3$ such that  $$d(f^{m_3k}(z_{3j}),f^{m_1k}(y_{1}))> \frac{a}{2}$$
for every $j=1,2,...n-1$ and that there exist at least $n-1$ points $\{z_{31}',z_{32}',...,z_{3(n-1)}'\}$ of $A_3$ such that  $$d(f^{m_3k}(z_{3j}'),f^{m_1k}(y_{2}))> \frac{a}{2}$$
for every $j=1,2,...n-1$, which implies that there exist at least $n-2$ points $\{z_{31}'',z_{32}'',...,z_{3(n-2)}''\}$ of $A_3$ such that
\begin{equation}\label{lem2}
\begin{aligned}
d(f^{m_3k}(z_{3j}''),f^{m_1k}(y_{1}))> \frac{a}{2}
\end{aligned}
\end{equation}
and
\begin{equation}\label{lem3}
\begin{aligned}
d(f^{m_3k}(z_{3j}''),f^{m_1k}(y_{2}))> \frac{a}{2}
\end{aligned}
\end{equation}
for every $j=1,2,...n-2$. Now take $y_3\in \{z_{31}'',z_{32}'',...,z_{3(n-2)}''\}$. By (\ref{lem1}), (\ref{lem2}) and (\ref{lem3}), we have
$$\min_{1\leq j<l\leq 3}d(f^{m_jk}(y_{j}),f^{m_lk}(y_{l}))> \frac{a}{2}.$$

 Continuing like this, after $n-2$ steps we have that there exist $y_i\in A_i$ for each $i\in\{1,2,...,n-1\}$ such that
\begin{equation}\label{lem4}
\begin{aligned}
\min_{1\leq j<l\leq n-1}d(f^{m_jk}(y_{j}),f^{m_lk}(y_{l}))> \frac{a}{2}.
\end{aligned}
\end{equation}

Step $n-1$. Since $$\min_{1\leq j<l\leq n}d(f^{m_nk}(x_{nj}),f^{m_nk}(x_{nl}))> a,$$
there exist at least $n-1$ points $\{z_{31}^{i},z_{32}^{i},...,z_{3(n-1)}^{i}\}$ of $A_n$ such that  $$d(f^{m_nk}(z_{3j}^{i}),f^{m_1k}(y_{i}))> \frac{a}{2}$$
for every $j=1,2,...n-1$ and every $i=1,2,...,n-1$, which implies that there exists $y_n\in A_n$ such that
\begin{equation}\label{lem5}
\begin{aligned}
d(f^{m_nk}(y_{n}),f^{m_lk}(y_{l}))> \frac{a}{2}
\end{aligned}
\end{equation}
for each $l\in\{1,2,...,n-1\}$. By (\ref{lem4}) and (\ref{lem5}), we have that there exist $y_i\in A_i$ for each $i\in\{1,2,...,n\}$ such that
$$\min_{1\leq j<l\leq n}d(f^{m_jk}(y_{j}),f^{m_lk}(y_{l}))> \frac{a}{2}.$$

\end{proof}

\begin{lemma}\label{yinli2}
	Let $n\in\mathbb{N}$ with $n\geq2$ and let $t_i\in\mathbb{N}$, put $p_i=t_1+t_2+...+t_i+2$ where $i=1,2,...,n$. Let $A_i=\{x_{i1},x_{i2},...,x_{ip_i}\}$, where $i=1,2,...,n$. If for each $i\in\{1,2,...,n\}$,
$$\min_{1\leq j<l\leq p_i}d(f^{m_ik}(x_{ij}),f^{m_ik}(x_{il}))> a,$$
then there exist $z_{i1},...,z_{it_i}\in A_i$ for each $i\in\{1,2,...,n\}$ such that
$$\min_{jr\neq lq}d(f^{m_jk}(z_{jr}),f^{m_lk}(z_{lq}))> \frac{a}{2}$$
where $1\leq j,l\leq n, 1\leq r\leq t_j, 1\leq q\leq t_l$.

\end{lemma}
\begin{proof}
The proof is similar to Lemma \ref{yinli1}.

Step 1. Since
$$\min_{1\leq j<l\leq p_1}d(f^{m_1k}(x_{1j}),f^{m_1k}(x_{1l}))> a,$$
there exist $t_1$ points $\{z_{11},z_{12},...,z_{1t_1}\}$ of $A_1$ such that
\begin{equation}\label{lem6}
\begin{aligned}
\min_{1\leq j<l\leq t_1}d(f^{m_1k}(z_{1j}),f^{m_1k}(z_{1l}))>  \frac{a}{2}.
\end{aligned}
\end{equation}

Step 2. Since
\begin{equation}\label{lem8}
\begin{aligned}
\min_{1\leq j<l\leq p_2}d(f^{m_2k}(x_{2j}),f^{m_2k}(x_{2l}))> a,
\end{aligned}
\end{equation}
there exist at least $p_2-1$ points $\{w_{21}^{i},w_{22}^{i},...,w_{2(p_2-1)}^{i}\}$ of $A_2$ such that  $$d(f^{m_2k}(w_{2j}^{i}),f^{m_1k}(z_{1i}))> \frac{a}{2}$$
for every $j=1,2,...p_2-1$ and every $i=1,2,...,t_1$, which implies that there exist $t_2$ points $\{z_{21},z_{22},...,z_{2t_2}\}$ of $A_2$ such that
\begin{equation}\label{lem7}
\begin{aligned}
\min_{1\leq l\leq t_1,1\leq j\leq t_2}d(f^{m_2k}(z_{2j}),f^{m_1k}(z_{1l}))>  \frac{a}{2}.
\end{aligned}
\end{equation}
By (\ref{lem6}), (\ref{lem7})and (\ref{lem8}), we have
$$\min_{jr\neq lq}d(f^{m_jk}(z_{jr}),f^{m_lk}(z_{lq}))> \frac{a}{2},$$
where $1\leq j,l\leq 2, 1\leq r\leq t_j, 1\leq q\leq t_l$.

Continuing like this, after $n-2$ steps we have that there exist  $z_{i1},...,z_{it_i}\in A_i$ for each $i\in\{1,2,...,n-1\}$ such that
\begin{equation}\label{lem9}
\begin{aligned}
\min_{jr\neq lq}d(f^{m_jk}(z_{jr}),f^{m_lk}(z_{lq}))> \frac{a}{2}
\end{aligned}
\end{equation}
where $1\leq j,l\leq n-1, 1\leq r\leq t_j, 1\leq q\leq t_l$.

Step $n-1$. Since
\begin{equation}\label{lem11}
\begin{aligned}
\min_{1\leq j<l\leq p_n}d(f^{m_nk}(x_{nj}),f^{m_nk}(x_{nl}))> a,
\end{aligned}
\end{equation}
there exist at least $p_n-1$ points $\{w_{n1}^{lr},w_{n2}^{lr},...,z_{n(t_n-1)}^{lr}\}$ of $A_n$ such that  $$d(f^{m_nk}(w_{nj}^{lr}),f^{m_lk}(z_{lr}))> \frac{a}{2}$$
for every $j=1,2,...,p_n-1$ and every $l=1,2,...,n-1,r=1,2,...,t_l$, which implies that there exist $z_{n1},...,z_{nt_n}\in A_n$ such that
\begin{equation}\label{lem10}
\begin{aligned}
\min_{1\leq j\leq p_n-1,1\leq l\leq n-1,1\leq r\leq t_l}d(f^{m_nk}(z_{nj}),f^{m_lk}(z_{lr}))> \frac{a}{2}.
\end{aligned}
\end{equation}
 By (\ref{lem9}), (\ref{lem11}) and (\ref{lem10}), we have that $$\min_{jr\neq lq}d(f^{m_jk}(z_{jr}),f^{m_lk}(z_{lq}))> \frac{a}{2}$$
where $1\leq j,l\leq n, 1\leq r\leq t_j, 1\leq q\leq t_l$.

\end{proof}

\section{Main results}

\begin{theorem}\label{th1}
Let $X$ be locally connected, $n\in\mathbb{N}$ with $n\geq2$ and $\beta=(m_1,m_2,...,m_n)\in \mathbb{N}^n$. If $(X,f)$ is multi-sensitive with respect to $\beta$, then it is $\beta$-$n$-sensitive.  Converse is not true in general.

\end{theorem}
\begin{proof}
Let $\beta=(m_1,m_2,...,m_n)\in \mathbb{N}^n$. For any nonempty open set $U\subset X$, since $X$ is locally connected, there exists a connected nonempty open subset $V\subset U$. Since $(X,f)$ is multi-sensitive with respect to $\beta$ with the sensitive constant $\delta$, $N_{f^{m_1}}(V,\delta)\cap N_{f^{m_2}}(V,\delta)\cap\cdot\cdot\cdot\cap N_{f^{m_n}}(V,\delta)\neq\emptyset$. Therefore, there are $k\in\mathbb{N}$ and $x_1,x_2,...,x_n,y_1,y_2,...,y_n\in V$ such that $d(f^{m_ik}(x_i),f^{m_ik}(y_i))>\delta$ for each $i=1,2,...,n$. Denote $\delta_i=d(f^{m_ik}(x_i),f^{m_ik}(y_i))$ for each $i=1,2,...,n$ and $\delta_0=\min_{0\leq i\leq n}\{\delta_i\}$ . Note that $\delta_0>\delta$. Let $g_i:f^{m_ik}(V)\rightarrow R$ with $g_i(z)=d(f^{m_ik}(x_i),z)$, where $i=1,2,...,n$. As $g_i$ is continuous and $f^{m_ik}(V)$ is connected, $g_i(f^{m_ik}(x_i))=0$ and $g_i(f^{m_ik}(y_i))=\delta_i$, we have $f^{m_ik}(V)\supset[0,\delta_i]$. So there are $n$ distinct points $x_{i1}'=f^{m_ik}(x_i),x_{i2}',...,x_{in}'=f^{m_ik}(y_i)\in f^{m_ik}(V)$ such that $d(f^{m_ik}(x_{i1}'),f^{m_ik}(x_{ij}'))=\frac{j-1}{n-1}\delta_i$ for each $j=1,2,...,n$. Take $x_{i1}=x_i,x_{i2},...,x_{in}=y_i\in V$ with $f^{m_ik}(x_{ij})=x_{ij}'$ for each $j=1,2,...,n$. Then we have $$\min_{1\leq j<l\leq n}d(f^{m_ik}(x_{ij}),f^{m_ik}(x_{il}))\geq \frac{\delta_i}{n-1}.$$
Thus, for each $i=1,2,...,n$, we have
$$\min_{1\leq j<l\leq n}d(f^{m_ik}(x_{ij}),f^{m_ik}(x_{il}))> \frac{\delta}{n}.$$
By Lemma \ref{yinli1}, there exist $x_i'\in\{x_{i1},x_{i2},...,x_{in}\}$ for each $i=1,2,...,n$ such that
$$\min_{1\leq j<l\leq n}d(f^{m_jk}(x_j'),f^{m_lk}(x_l')> \frac{\delta}{2n},$$
which shows that $(X,f)$ is $\beta$-$n$-sensitive.

Conversely. Then we know that there is a dynamical system $(X,f)$ such that it is sensitive but not multi-sensitive, besides, $X$ is locally connected. Since the dynamical system $(X,f)$ is not multi-sensitive, there exists $k\in \mathbb{N}$ such that $(X,f)$ is not multi-sensitive with $\beta$, where $\beta=(1,1,...,1)\in \mathbb{N}^k$. On the other hand, since $X$ is locally connected and $(X,f)$  is sensitive, then it is $k$-sensitive \cite{a1}, further it is $\beta$-$k$-sensitive.
\end{proof}

\begin{theorem}\label{th2}
Let $X$ be locally connected. If $(X,f)$ is $\mathcal{N}$ sensitive, then it is strongly $\beta$-$n$-sensitive.

\end{theorem}
\begin{proof}
Let $n\in \mathbb{N}$ with $n\geq2$ and $\beta=(m_1,m_2,...,m_n)\in \mathbb{N}^n$. Denote $\alpha=(\alpha_1,\alpha_2,...,\alpha_n)$ is replace of $\beta$ with $\alpha_i\leq\alpha_j$ if $i<j$.
Let
$$\alpha=(\underbrace{\alpha_{t_1},\alpha_{t_1},...,\alpha_{t_1}}_{k_1-fold},\underbrace{\alpha_{t_2},\alpha_{t_2},...,\alpha_{t_2}}_{k_2-fold},...,\underbrace{\alpha_{t_s},\alpha_{t_s},...,\alpha_{t_s}}_{k_s-fold}),$$
where $1=t_1<t_2<...<t_s\leq n$ and $k_1+k_2+...+k_s=n$.
For any nonempty open set $U\subset X$, since $X$ is locally connected, there exists a connected nonempty open subset $V\subset U$. Since $(X,f)$ is $\mathcal{N}$ sensitive with the sensitive constant $\delta$, $N_{f}(V,\delta)\cap N_{f^2}(V,\delta)\cap...\cap N_{f^{\alpha_{t_s}}}(V,\delta)\neq\emptyset$. Therefore, there are $k\in\mathbb{N}$ and $x_1,x_2,...,x_s,y_1,y_2,...,y_s\in V$ such that $d(f^{\alpha_{t_i}k}(x_i),f^{\alpha_{t_i}k}(y_i))>\delta$ for each $i=1,2,...,s$. Denote $\delta_i=d(f^{\alpha_{t_i}k}(x_i),f^{\alpha_{t_i}k}(y_i))$ for each $i=1,2,...,s$ and $\delta_0=\min_{0\leq i\leq n}\{\delta_i\}$ . Note that $\delta_0>\delta$. Let $g_i:f^{\alpha_{t_i}k}(V)\rightarrow R$ with $g_i(z)=d(f^{\alpha_{t_i}k}(x_i),z)$, where $i=1,2,...,s$. As $g_i$ is continuous and $f^{\alpha_{t_i}k}(V)$ is connected, $g_i(f^{\alpha_{t_i}k}(x_i))=0$ and $g_i(f^{\alpha_{t_i}k}(y_i))=\delta_i$, we have $f^{\alpha_{t_i}k}(V)\supset[0,\delta_i]$. So there are $k_1+k_2+...+k_i+2$ distinct points $x_{i1}'=f^{\alpha_{t_i}k}(x_i),x_{i2}',...,x_{i[k_1+k_2+...+k_i+2]}'=f^{\alpha_{t_i}k}(y_i)\in f^{\alpha_{t_i}k}(V)$ such that $d(f^{\alpha_{t_i}k}(x_{i1}'),f^{\alpha_{t_i}k}(x_{ij}'))=\frac{j-1}{k_1+k_2+...+k_i+1}\delta_i$ for each $j=1,2,...,k_1+k_2+...+k_i+2$. Take $x_{i1}=x_i,x_{i2},...,x_{i[k_1+k_2+...+k_i+2]}=y_i\in V$ with $f^{\alpha_{t_i}k}(x_{ij})=x_{ij}'$ for each $j=1,2,...,k_1+k_2+...+k_i+2$. Then we have $$\min_{1\leq j<l\leq n}d(f^{\alpha_{t_i}k}(x_{ij}),f^{\alpha_{t_i}k}(x_{il}))\geq \frac{\delta_i}{k_1+k_2+...+k_i+1}> \frac{\delta_i}{n+2}.$$
Thus, for each $i=1,2,...,s$, we have
$$\min_{1\leq j<l\leq n}d(f^{\alpha_{t_i}k}(x_{ij}),f^{\alpha_{t_i}k}(x_{il}))> \frac{\delta}{n+2}.$$
By Lemma \ref{yinli2}, there exist $x_i^1,x_i^2,...,x_i^{k_i}\in\{x_{i1},x_{i2},...,x_{i[k_1+k_2+...+k_i+2]}\}$ for each $i=1,2,...,s$ such that
$$\min_{x_j^r\neq x_l^q}d(f^{m_jk}(x_j^r),f^{m_lk}(x_l^q))> \frac{\delta}{2n+4}$$
where $1\leq j,l\leq s, 1\leq r\leq t_j, 1\leq q\leq t_l$, which shows that $(X,f)$ is $\alpha$-$n$-sensitive. By definition of $\beta$-$n$-sensitive, $(X,f)$ is also $\beta$-$n$-sensitive. For the arbitrariness of $n,\beta$, it follows that $(X,f)$ is strongly $\beta$-$n$-sensitive.

\end{proof}

\begin{theorem}\label{th3}
Let $X$ be locally connected, $n,l\in\mathbb{N}$ with $n,l\geq2$, $\beta=(m_1,m_2,...,m_n)\in \mathbb{N}^n$ and $\alpha=(c_1,c_2,...,c_l)\in \mathbb{N}^l$. Denote $B=\{m_1,m_2,...,m_n\}$ and $K=\{k_1,k_2,...,k_l\}$. Suppose that $B\subset K$. If $(X,f)$ is multi-sensitive with respect to $\alpha$, then $(X,f)$ is $\beta$-$n$-sensitive.
\end{theorem}
\begin{proof}
Denote $\beta'=(b_1,b_2,...,b_n)$ is replace of $\beta$ with $b_i\leq b_j$ if $i<j$ and $\alpha'=(a_1,a_2,...,a_l)$ is replace of $\alpha$ with $a_i\leq a_j$ if $i<j$.

Let
$$\beta_1=(\underbrace{b_{t_1},b_{t_1},...,b_{t_1}}_{k_1-fold},\underbrace{b_{t_2},b_{t_2},...,b_{t_2}}_{k_2-fold},...,\underbrace{b_{t_s},b_{t_s},...,b_{t_s}}_{k_s-fold}),$$
where $1=t_1<t_2<...<t_s\leq n$ and $k_1+k_2+...+k_s=n$.
Let
$$\alpha_1=(a_{t_1},a_{t_2},...,a_{t_r})$$
where  $t_r= \sharp K$ and $a_{t_i}<a_{t_j}$ with $i<j$, ($\sharp A$ denotes the cardinality of the set $A$).
Since $B\subset K$, for each $b_{t_i}\in B$, there exists $a_{t_j}\in K$ such that $b_{t_i}=a_{t_j}$, where $i\in\{1,2,...,s\}$ and  $j\in\{1,2,...,r\}$.
For any nonempty open set $U\subset X$, since $X$ is locally connected, there exists a connected nonempty open subset $V\subset U$. Since $(X,f)$  is multi-sensitive with respect to $\alpha$ with the sensitive constant $\delta$, $(X,f)$ is also multi-sensitive with respect to $\alpha_1$ with the sensitive constant $\delta$, then $N_{f^{a_{t_1}}}(V,\delta)\cap N_{f^{a_{t_2}}}(V,\delta)\cap\cdot\cdot\cdot\cap N_{f^{a_{t_r}}}(V,\delta)\neq\emptyset$. Therefore, there are $k\in\mathbb{N}$ and $x_1,x_2,...,x_s,y_1,y_2,...,y_s\in V$ such that $d(f^{b_{t_i}k}(x_i),f^{b_{t_i}k}(y_i))>\delta$ for each $i=1,2,...,s$. Denote $\delta_i=d(f^{b_{t_i}k}(x_i),f^{b_{t_i}k}(y_i))$ and $\delta_0=min\{d_i\}$ for each $i=1,2,...,s$. Note that $\delta_0>\delta$. Let $g_i:f^{b_{t_i}k}(V)\rightarrow R$ with $g_i(z)=d(f^{b_{t_i}k}(x_i),z)$, where $i=1,2,...,s$. As $g_i$ is continuous and $f^{b_{t_i}k}(V)$ is connected, $g_i(f^{b_{t_i}k}(x_i))=0$ and $g_i(f^{b_{t_i}k}(y_i))=\delta_i$, we have $f^{b_{t_i}k}(V)\supset[0,b_i]$. So there are $k_1+k_2+...+k_i+2$ distinct points $x_{i1}'=f^{b_{t_i}k}(x_i),x_{i2}',...,x_{i[k_1+k_2+...+k_i+2]}'=f^{b_{t_i}k}(y_i)\in f^{\alpha_{t_i}k}(V)$ such that $d(f^{b_{t_i}k}(x_{i1}'),f^{b_{t_i}k}(x_{ij}'))=\frac{j-1}{k_1+k_2+...+k_i+1}\delta_i$ for each $j=1,2,...,k_1+k_2+...+k_i+2$. Take $x_{i1}=x_i,x_{i2},...,x_{i[k_1+k_2+...+k_i+2]}=y_i\in V$ with $f^{b_{t_i}k}(x_{ij})=x_{ij}'$ for each $j=1,2,...,k_1+k_2+...+k_i+2$. Then we have $$\min_{1\leq j<l\leq n}d(f^{b_{t_i}k}(x_{ij}),f^{b_{t_i}k}(x_{il}))\geq \frac{\delta_i}{k_1+k_2+...+k_i+1}> \frac{\delta_i}{n+2}.$$
Thus, for each $i=1,2,...,s$, we have
$$\min_{1\leq j<l\leq n}d(f^{b_{t_i}k}(x_{ij}),f^{b_{t_i}k}(x_{il}))> \frac{\delta}{n+2}.$$
By Lemma \ref{yinli2}, there exist $x_i^1,x_i^2,...,x_i^{k_i}\in\{x_{i1},x_{i2},...,x_{i[k_1+k_2+...+k_i+2]}\}$ for each $i=1,2,...,s$ such that
$$\min_{x_j^r\neq x_l^q}d(f^{m_jk}(x_j^r),f^{m_lk}(x_l^q))> \frac{\delta}{2n+4}$$
where $1\leq j,l\leq s, 1\leq r\leq t_j, 1\leq q\leq t_l$, which shows that $(X,f)$ is $\beta-n-$sensitive.
\end{proof}

\begin{theorem}
Let $X$ be locally connected, $n,l\in\mathbb{N}$ with $n,l\geq2$, $\beta=(m_1,m_2,...,m_n)\in \mathbb{N}^n$ and $\alpha=(c_1,c_2,...,c_l)\in \mathbb{N}^l$. Denote $B=\{m_1,m_2,...,m_n\}$ and $K=\{k_1,k_2,...,k_l\}=\{p_1,p_2,...,p_r\}$, where $p_i\neq p_j$ with $i\neq j$. Let $\gamma=(p_1,p_2,...,p_r,p_1,p_2,...,p_r)$. Suppose that $B\subset K$. If $(X,f)$ is $\gamma$-$2r$-sensitive, then $(X,f)$ is $\beta$-$n$-sensitive.
\end{theorem}
\begin{proof}

Denote $\beta'=(b_1,b_2,...,b_n)$ is replace of $\beta$ with $b_i\leq b_j$ if $i<j$.

Let
$$\beta_1=(\underbrace{b_{t_1},b_{t_1},...,b_{t_1}}_{k_1-fold},\underbrace{b_{t_2},b_{t_2},...,b_{t_2}}_{k_2-fold},...,\underbrace{b_{t_s},b_{t_s},...,b_{t_s}}_{k_s-fold}),$$
where $1=t_1<t_2<...<t_s\leq n$ and $k_1+k_2+...+k_s=n$.
Let
$$\alpha_1=(a_1,a_2,...,a_r)$$
where $a_i<a_j$ if $i<j$ and $\{a_1,a_2,...,a_r\}=\{p_1,p_2,...,p_r\}$. Since $B\subset K$, for each $b_{t_i}\in B$, there exists $a_{j}\in K$ such that $b_{t_i}=a_{j}$, where $i\in\{1,2,...,s\}$ and  $j\in\{1,2,...,r\}$.
For any nonempty open set $U\subset X$, since $X$ is locally connected, there exists a connected nonempty open subset $V\subset U$. Denote $\gamma=(p_1,p_2,...,p_r,p_1,p_2,...,p_r)=(d_1,d_2,...,d_r,d_{r+1},d_{r+2},...,d_{2r})$. Since $(X,f)$  is $\gamma$-$2r$-sensitive with the sensitive constant $\delta$,
there exist $k\in\mathbb{N}$ and  $x_1,x_2,...,x_{2r}\in V$ such that $$\min_{1\leq i< j\leq 2r}d(f^{d_ik}(x_i),f^{d_jk}(x_j))>\delta,$$
which implies that $\min_{1\leq i\leq r}d(f^{a_ik}(x_i),f^{a_ik}(x_i))>\delta$. Hence, $d(f^{b_{t_i}k}(x_i),f^{b_{t_i}k}(y_i))\geq\delta$ for each $i=1,2,...,s$. Denote $\delta_i=d(f^{b_{t_i}k}(x_i),f^{b_{t_i}k}(y_i))$ and $\delta_0=min\{d_i\}$ for each $i=1,2,...,s$. Note that $\delta_0>\delta$. Let $g_i:f^{b_{t_i}k}(V)\rightarrow R$ with $g_i(z)=d(f^{b_{t_i}k}(x_i),z)$, where $i=1,2,...,s$. As $g_i$ is continuous and $f^{b_{t_i}k}(V)$ is connected, $g_i(f^{b_{t_i}k}(x_i))=0$ and $g_i(f^{b_{t_i}k}(y_i))=\delta_i$, we have $f^{b_{t_i}k}(V)\supset[0,b_i]$. So there are $k_1+k_2+...+k_i+2$ distinct points $x_{i1}'=f^{b_{t_i}k}(x_i),x_{i2}',...,x_{i[k_1+k_2+...+k_i+2]}'=f^{b_{t_i}k}(y_i)\in f^{\alpha_{t_i}k}(V)$ such that $d(f^{b_{t_i}k}(x_{i1}'),f^{b_{t_i}k}(x_{ij}'))=\frac{j-1}{k_1+k_2+...+k_i+1}\delta_i$ for each $j=1,2,...,k_1+k_2+...+k_i+2$. Take $x_{i1}=x_i,x_{i2},...,x_{i[k_1+k_2+...+k_i+2]}=y_i\in V$ with $f^{b_{t_i}k}(x_{ij})=x_{ij}'$ for each $j=1,2,...,k_1+k_2+...+k_i+2$. Then we have $$\min_{1\leq j<l\leq n}d(f^{b_{t_i}k}(x_{ij}),f^{b_{t_i}k}(x_{il}))\geq \frac{\delta_i}{k_1+k_2+...+k_i+1}> \frac{\delta_i}{n+2}.$$
Thus, for each $i=1,2,...,s$, we have
$$\min_{1\leq j<l\leq n}d(f^{b_{t_i}k}(x_{ij}),f^{b_{t_i}k}(x_{il}))> \frac{\delta}{n+2}.$$
By Lemma \ref{yinli2}, there exist $x_i^1,x_i^2,...,x_i^{k_i}\in\{x_{i1},x_{i2},...,x_{i[k_1+k_2+...+k_i+2]}\}$ for each $i=1,2,...,s$ such that
$$\min_{x_j^r\neq x_l^q}d(f^{m_jk}(x_j^r),f^{m_lk}(x_l^q))> \frac{\delta}{2n+4}$$
where $1\leq j,l\leq s, 1\leq r\leq t_j, 1\leq q\leq t_l$, which shows that $(X,f)$ is $\beta-n-$sensitive.

\end{proof}

\begin{theorem}\label{th5}
 Let $n,l\in\mathbb{N}$ with $n,l\geq2$, $\beta=(m_1,m_2,...,m_n)\in \mathbb{N}^n$ and $\alpha=(k_1,k_2,...,k_l)\in \mathbb{N}^l$. Assume that for each $i=1,2,...,n$, there exists different $j\in\{1,2,...,l\}$ such that $m_i=k_j$. If $(X,f)$ is $\alpha$-$l$-sensitive, then $(X,f)$ is $\beta$-$n$-sensitive.
\end{theorem}
\begin{proof}
Let $U\subseteq X$ be a nonempty set and $\beta=(m_1,m_2,...,m_n)\in \mathbb{N}^n$. Since $(X,f)$ is $\alpha$-$l$-sensitive, there are $\delta>0$, $p\in\mathbb{N}$ and $x_1,x_2,...,x_l\in U$ such that
$$\min_{1\leq i<j\leq l}d(f^{k_ip}(x_i),f^{k_jp}(x_j))>\delta.$$
By the assumption, $$\min_{1\leq i<j\leq n}d(f^{m_ip}(x_i),f^{m_jp}(x_j))>\delta,$$
which shows that $(X,f)$ is $\beta$-$n$-sensitive.

\end{proof}

The following two examples are the application of Theorem \ref{th5}. Besides, them show that if we drop the  condition of Theorem \ref{th5}, then the result does hold in general. Moreover, them show that the $\beta$-$n$-sensitivity is not iteration invariants in general.

\begin{example}\label{ex1}

Let $X=\{1,2,3,4\}$ and $f:X\rightarrow X$ be the map defined by $f(1)=2$,$f(2)=3$,$f(3)=4$,$f(4)=1$.

Take $\delta=\frac{1}{2}$. For $n=3$ and $\beta=(2,3,4)$, it is is easy to verify that for any nonempty set $U\subseteq X$, there exist $x_1=x_2=x_3\in U$ and $k=1$ such that
$$\min_{1\leq i<j\leq 3}d(f^{m_i}(x_i),f^{m_j}(x_j))>\frac{1}{2}.$$
So it is $\beta$-$3$-sensitive.

But for $n=4$ and $\beta_1=(2,3,3,4)$, it is is easy to verify that $(X,f)$ is not $\beta_1$-$4$-sensitive.
In fact, for $U=\{1\}$ and any $m\in\mathbb{N}$, $d(f^{3m}(1),f^{3m}(1))\equiv 0$. This shows that $(2,3,4)$-$3$-sensitivity does not imply $(2,3,3,4)$-$4$-sensitivity.

Besides, for $n=2$ and $\beta_2=(2,4)$, with the similar argument, $(X,f)$ is $\beta_2$-$2$-sensitive, which shows that $(2,3,4)$-$3$-sensitivity implies $(2,4)$-$2$-sensitivity. It is also the result of Theorem \ref{th5}.

However, for $n=2$ and $\beta_3=(4,8)$ and any $\lambda>0$, there is a nonempty set $\{1\}$ such that for any $k\in \mathbb{N}$
$$d(f^{4k}(1),f^{8k}(1))\equiv0,$$
which shows that $(X,f)$ is not $\beta_3$-$2$-sensitive. Further, $(X,f^2)$ is not $\beta_2$-$2$-sensitive. It shows that $(2,4)$-$2$-sensitivity does not imply $(4,8)$-$2$-sensitivity. Besides, it shows that $(X,f)$ is $(2,4)$-$2$-sensitive but $(X,f^2)$ is not $(2,4)$-$2$-sensitive. Therefore, the $\beta$-$n$-sensitivity is not preserved under iterations  in general.

\end{example}

In the above example $X$ is finite, but in the next example $X$ is  infinite and has no isolated points.

\begin{example}

Let $f:S^1\rightarrow S^1$ be the map defined by $f(e^{2\pi i\theta})=e^{2\pi i(\theta+\frac{1}{8})}$, where $S^1$ is a unit circle on the complex plane, $\theta\in [0,1]$.

With the similar argument, it clear that $(S^1,f)$ is $(2,3,4)$-$3$-sensitive and $(2,4)$-$2$-sensitive but neither $(2,3,3,4)$-$4$-sensitive nor $(8,16)$-$2$-sensitive. Besides, it is obvious that $(S^1,f)$ is not sensitive, which shows that there is not relationship between the $\beta$-$n$-sensitivity and sensitivity.

\end{example}

\begin{theorem}\label{thsy}
Let $n\in \mathbb{N}$ with $n\geq2$. If $(X,f)$ is thickly syndetically $n$-sensitive, then it is semi-strongly $\beta$-$n$-sensitive.

\end{theorem}
\begin{proof}
Let $\beta=(\beta_1,\beta_2,...,\beta_n)\in \mathbb{N}^n$. For any nonempty open set $U\subset X$, since $(X,f)$ is thickly syndetically $n-$sensitive with the sensitive constant $\delta$, there exist thickly syndetical set $A$ and $n$ points $x_1,x_2,...,x_n\in X$ such that $$\min_{1\leq i<j\leq n}d(f^{m_k}(x_{i}),f^{m_k}(x_{j}))> \delta$$ for any $m_k\in A$.
For each $l=1,2,...,n$, take $\alpha_l>0$ such that $d(a,b)\leq\alpha_i$ implies $d(f^q(a),f^q(b))\leq \delta$ for any $a,b\in X$ and $0\leq q\leq m_l-1$. Denote $\lambda=\min_{1\leq l\leq n}\{\alpha_l\}$. We claim that $\frac{\lambda}{2}$ does the job. Firstly, we may assume that $m_k>\max_{1\leq l\leq n}\{\beta_l\}$ for any $m_k\in A$. Then for each $l=1,2,...,n$ and each $k\in \mathbb{N}$ there exist $p_{lk},j_{lk}$ such that $m_k=\beta_lp_{lk}+j_{lk}$, where $p_{lk}\in \mathbb{N}, j_{lk}\in\{1,2,...,\beta_l-1\}$.
Thus, we have $$\min_{1\leq i<j\leq n}d(f^{\beta_lp_{lk}}(x_{i}),f^{\beta_lp_{lk}}(x_{j}))>\alpha_i$$ for each $l=1,2,...,n$ and each $k\in \mathbb{N}$. Further, $$\min_{1\leq i<j\leq n}d(f^{\beta_ip_{ik}}(x_{i}),f^{\beta_ip_{ik}}(x_{j}))>\lambda$$ for each $i=1,2,...,n$ and each $k\in \mathbb{N}$. Since $\{m_k\}_{k=1}^\infty$ is thickly syndetical set, $\{p_{lk}\}_{k=1}^\infty$ is also thickly syndetical set for each $l=1,2,...,n$. Then $\cap_{1\leq l\leq n}\{p_{lk}\}_{k=1}^\infty$ is thickly syndetical set. Take $r\in \cap_{1\leq l\leq n}\{p_{lk}\}_{k=1}^\infty$. Then $$\min_{1\leq i<j\leq n}d(f^{r\beta_l}(x_{i}),f^{r\beta_l}(x_{j}))>\lambda.$$
By Lemma \ref{yinli1},
$$\min_{1\leq i<j\leq n}d(f^{r\beta_i}(x_{i}),f^{r\beta_j}(x_{j}))>\frac{\lambda}{2},$$
which shows that $(X,f)$ is $\beta$-$n$-sensitive.  For the arbitrariness of $\beta$,  it
follows that $(X,f)$ is semi-strongly $\beta$-$n$-sensitive.
\end{proof}
From the proof of the Theorem \ref{thsy}, we can get the following stronger results.
\begin{corollary}\label{cos}
Let $n\in \mathbb{N}$ with $n\geq2$.  If $(X,f)$ is thickly syndetically $n$-sensitive, then it is thickly syndetically semi-strongly $\beta$-$n$-sensitive.
\end{corollary}

\begin{corollary}\label{co1}
Let $X$ be locally connected. Then $(X,f)$ is thickly syndetically sensitive if and only if it is  thickly syndetically strongly $\beta$-$n$-sensitive.
\end{corollary}
\begin{proof}
For any $n\in \mathbb{N}$, since $(X,f)$ is thickly syndetically sensitive and locally connected, $(X,f)$ is thickly syndetically $n$-sensitive with the similar argument in Theorem \ref{th1}. Then by Corollary \ref{cos}, we get the result.

Conversely, it is clear by definitions.
\end{proof}

\begin{corollary}\label{coc}
Let $n\in \mathbb{N}$ with $n\geq2$. If $(X,f)$ is cofinitely $n$-sensitive, then it is  thickly syndetically semi-strongly $\beta$-$n$-sensitive.
\end{corollary}
\begin{proof}
Since $(X,f)$ is cofinitely $n$-sensitive, $(X,f)$ is thickly syndetically $n$-sensitive. Then by Corollary \ref{cos}, we get the result.

\end{proof}

\begin{theorem}
Let $X$ be locally connected. Then $(X,f)$ is cofinitely sensitive if and only if it is cofinitely strongly $\beta$-$n $-sensitive.
\end{theorem}
\begin{proof}
For any $n\in \mathbb{N}$, since $(X,f)$ is cofinitely sensitive and locally connected, $(X,f)$ is cofinitely $n$-sensitive with the similar argument in Theorem \ref{th1}. Then by Corollary \ref{coc}, we get the result.

Conversely, it is clear by definitions.

\end{proof}

The following theorem  improve and extend the \cite[Theorem 3.1]{a1} because of the thickly syndetical transitivity strictly weaken than topological mixing.
\begin{theorem}\label{thtr}
If $(X,f)$ is thickly syndetically transitive, then it is strongly $\beta$-$n$-sensitive.
\end{theorem}

\begin{proof}
By \cite[Proposition 3.3]{a2}, if $(X,f)$ is thickly syndetically transitive, then it is thickly syndetically $n-$sensitive for any $n\geq2$. Using Theorem \ref{thsy}, we can complete the proof.

\end{proof}

\begin{corollary}
If $(X,f)$ is thickly syndetically transitive, then it is  thickly syndetically strongly $\beta$-$n$-sensitive.
\end{corollary}
\begin{proof}
 With the similar proof of Corollary \ref{coc} and Theorem \ref{thtr}, we get the result.

\end{proof}

\begin{theorem}
Let $(X,f)$ be an $M-$system and Let $n\geq2$. If $(X,f)$ is thickly syndetically sensitive, then it is $\beta$-$n$-sensitive.
\end{theorem}

\begin{proof}
By \cite[Proposition 3.2]{a2}, if $(X,f)$ is thickly syndetically sensitive, then it is thickly syndetically $n-$sensitive. Using Theorem \ref{thsy}, we can complete the proof.

\end{proof}

The following theorem removes the condition `$f$ is surjective' in \cite[Corollary 3.2]{a1}.
\begin{theorem}
If $(X,f)$ is  chain-mixing and has  shadowing property, then it is strongly multi-$\beta$-$n$-sensitive.
\end{theorem}

\begin{proof}
By \cite[Theorem 3.6]{a3}, if $(X,f)$ is $(X,f)$ is  chain-mixing and has  shadowing property, then it is $\Delta$-mixing, and hence it is mixing. Then using  \cite[Theorem 3.1]{a1}, we can complete the proof.

\end{proof}

In \cite[Theorem 4.1]{b16}, it was proved that  if $(X,f)$ is multi-transitive and
the periodic set of $f$ is dense in $X$, then $(X,f)$ is $\mathcal{N}$-sensitive. On the one hand, by \cite{a4} and the fact that multi-transitivity implies total transitivity, multi-transitivity and the dense periodic set of $f$ imply weakly mixing. Besides, multi-transitivity and weakly mixing is equivalent to strong multi-transitivity \cite[Proposition 3.3]{a5}. Therefore, the condition of the following Theorem \ref{thmul} is weaker than the condition of \cite[Theorem 4.1]{b16}. On the other hand, strongly multi-sensitivity is stronger than $\mathcal{N}$-sensitivity. So the following Theorem \ref{thmul} improves the \cite[Theorem 4.1]{b16}.
\begin{theorem}\label{thmul}
 If $(X,f)$ is strongly multi-transitive, then it is strongly multi-sensitive.
\end{theorem}

\begin{proof}
Take nonempty sets $V_1,V_2\subseteq X$ and $\delta>0$ such that $d(V_1,V_2)>\delta$. For any $m\in \mathbb{N}$ and any vector $a=(a_1,a_2,...,a_m)$, denote $b=(a_1,a_2,...,a_m,a_1,a_2,...,a_m)$, and for any nonempty sets $U_1,U_2,...,U_m\subseteq X$, since $(X,f)$ is strongly multi-transitive, the product system $(X^{2m},f^b)$ is transitive, that is, for $(U_1,U_2,...,U_m,U_1,U_2,...,U_m)$ and $(\underbrace{V_1,V_1,...,V_1}_{m-fold},\underbrace{V_2,V_2,...,V_2}_{m-fold})$, there exists $k\in \mathbb{N}$ such that $f^{ka_i}(U_i)\cap V_1\neq\emptyset$ and $f^{ka_i}(U_i)\cap V_2\neq\emptyset$ for each $i\in\{1,2,...,m\}$. Therefore, for every $i\in\{1,2,...,m\}$, there exist $x_{i},y_{i}\in U_i$ such that $f^{ka_i}(x_i)\in V_1$ and  $f^{ka_i}(y_i)\in V_2$. Since $d(V_1,V_2)>\delta$, $d(f^{ka_i}(x_i),f^{ka_i}(y_i))>\delta$, which implies that
$$\bigcap_{i=1}^m N_{f^{a_i}}(U_i,\delta)\neq\emptyset.$$
Therefore $(X,f)$ is strongly multi-sensitive.

\end{proof}

\begin{corollary}
Let $X$ be locally connected and $(X,f)$ be a minimal system. If $(X,f)$ is weakly mixing, then it is  strongly $\beta$-$n$-sensitive.
\end{corollary}
\begin{proof}
From the \cite[Corollary 7]{a6}, it follows that weak mixing and multi-transitivity are
equivalent in a minimal system. Therefore, by Theorem \ref{thmul} and Theorem \ref{th2}, $(X,f)$ is  strongly $\beta$-$n$-sensitive.

\end{proof}

In  \cite[Corollary 3.4]{a1}, the authors proved that $n$-sensitivity is preserved under iterations. But in Example \ref{ex1}, it shows that $\beta$-$n$-sensitivity is not preserved under iterations in general. Now we consider the result for strong $\beta$-$n$-sensitivity and semi-strong $\beta$-$n$-sensitivity.
\begin{theorem}
For any $k\in\mathbb{N}$, $(X,f)$ is strongly $\beta$-$n$-sensitive (semi-strongly $\beta$-$n$-sensitive, respectively) if and only if $(X,f^k)$ is strongly $\beta$-$n$-sensitive (semi-strongly $\beta$-$n$-sensitive, respectively).
\end{theorem}

\begin{proof}
Sufficiency. Let $n\in \mathbb{N}$ and vector $\beta=(a_1,a_2,...,a_n)$. For any nonempty sets $U_1,U_2,...,U_n\subseteq X$, since $(X,f^k)$ is strongly $\beta-n-$sensitive with the sensitive constant $\delta$, there exists  $p\in \mathbb{N}$ such that
$$\min_{1\leq i<j\leq n}d(f^{ka_ip}(x_i),f^{ka_jp}(x_j))>\delta.$$
Take $q=kp$, then we have that
$$\min_{1\leq i<j\leq n}d(f^{a_iq}(x_i),f^{a_jq}(x_j))>\delta.$$
 Therefore $(X,f)$ is strongly $\beta$-$n$-sensitive.

Necessity. Let $n\in \mathbb{N}$ and vector $\beta=(a_1,a_2,...,a_n)$. Take $n_1=kn$ and $\beta_1=(ka_1,ka_2,...,ka_n)$. Since $(X,f)$ is strongly $\beta-n-$sensitive with the sensitive constant $\delta$, there exists  $p\in \mathbb{N}$ such that
$$\min_{1\leq i<j\leq n}d(f^{ka_ip}(x_i),f^{ka_jp}(x_j))>\delta.$$
Therefore $(X,f^k)$ is strongly $\beta$-$n$-sensitive.

With the similar argument, one can get that semi-strong $\beta$-$n$-sensitivity is also iteration invariant.

\end{proof}

Finally, we give a criteria for the existence of strong multi-$\beta$-$n$-sensitivity with respect to $\alpha$.
\begin{theorem}
 If $(X,f)$ is mixing, then it is strongly multi-$\beta$-$n$-sensitive with respect to $\alpha$.
\end{theorem}

\begin{proof}
Let $n\in \mathbb{N}$ with $n\geq2$. And let $\beta=(b_1,b_2,...,b_n)\in \mathbb{N}^n$ and $\alpha=(a_1,a_2,...,a_n)\in \mathbb{N}^n$. Take $n$ distinct points $z_1,z_2,...,z_n$ of $X$ and let $\min_{1\leq i<j\leq n}d(z_i,z_j)=2\delta$. Let $0<\varepsilon<\frac{\delta}{4}$ and $V_i=\bar{B(z_i,\varepsilon)}$ for each $i=1,2,...,n$. Then we can get that $\min_{1\leq i<j\leq n}d(V_i,V_j)>\delta$. Besides, let $U_1,U_2,...,U_m$ be nonempty open subsets in $X$. For each $U_l$, since $(X,f)$ is mixing, there exists a common integer $N_l\in \mathbb{N}$ such that for any $k>N_l$ satisfying $f^k(U_l)\cap V_i\neq\emptyset$ for all $i=1,2,...,n$, which implies that there are $x_{l1},x_{l2},...,x_{ln}\in U_l$ such that $f^k(x_{li})\in V_i$ for all $i=1,2,...,n$ and all $k>N_l$. And since $ka_ib_i>N_l$ for each $i=1,2,...,n$, then $\min_{1\leq i<j\leq n}d(f^{ka_ib_i}(x_{li}),f^{ka_jb_j}(x_{lj}))>\delta$ for all  $k>N_l$. Take $N=\max\{N_1,N_2,...,N_m\}$. Then for any $k>N$, we have $\min_{1\leq i<j\leq n}d(f^{ka_ib_i}(x_{li}),f^{ka_jb_j}(x_{lj}))>\delta$, that is, $k\in \cap_{i=1}^mN_{f^{a_i}}(U_i,\beta,\delta)$. Therefore, $(X,f)$ is  strongly multi-$\beta$-$n$-sensitive with respect to $\alpha$.

\end{proof}

\section{Conclusion}
This paper focuses on the relationships among the $\beta$-$n$-sensitivity and other forms of sensitivity (transitivity) in discrete dynamical systems. And we get the following results.

\textcircled{1} mixing, \textcircled{2} strong multi-$\beta$-$n$-sensitivity with respect to $\alpha$, \textcircled{3} thickly syndetical transitivity, \textcircled{4} thickly syndetically strong $\beta$-$n$-sensitivity, \textcircled{5} thickly syndetical sensitivity,
\textcircled{6} thickly syndetical $n$-sensitivity, \textcircled{7} semi-strong $\beta$-$n$-sensitivity, \textcircled{8} $M$-system and thickly syndetical sensitivity, \textcircled{9} cofinite sensitivity, \textcircled{10} cofinitely strong $n$-sensitivity,
\textcircled{11} cofinitely strong $\beta$-$n$-sensitivity, \textcircled{12} multi-sensitivity with respect to $\beta$, \textcircled{13} $\beta$-$n$-sensitivity, \textcircled{14} strong $\beta$-$n$-sensitivity, \textcircled{15} $\mathcal{N}$ sensitivity, \textcircled{16} strong multi-transitivity, \textcircled{17} strong multi-sensitivity. We denote locally connected $X$  by $lc$.

Then we have the following  relationships.

$$\textcircled{1} \Rightarrow \textcircled{3}\Rightarrow\textcircled{4}
\overset{\text{+lc}}{\Longleftrightarrow}\textcircled{5}\overset{\text{+lc}}{\Longrightarrow}\textcircled{6}\Rightarrow\textcircled{7}\Rightarrow\textcircled{13},$$

$$\textcircled{1} \Rightarrow \textcircled{9}\overset{\text{+lc}}{\Longrightarrow}\textcircled{11}
\overset{\text{+lc}}{\Longleftrightarrow}\textcircled{10}\overset{\text{+lc}}{\Longrightarrow}\textcircled{6}\Leftarrow\textcircled{8},$$

$$\textcircled{1} \Rightarrow \textcircled{2}\Rightarrow\textcircled{14}
\Rightarrow\textcircled{13}\overset{\text{+lc}}{\Longleftarrow}\textcircled{12},$$

$$\textcircled{1} \Rightarrow \textcircled{16}\Rightarrow\textcircled{17}
\Rightarrow\textcircled{15}\overset{\text{+lc}}{\Longrightarrow}\textcircled{14},$$

\section{Problem}

Does weakly mixing imply $\beta$-$n$-sensitive?


\end{document}